\documentclass[11pt]{article}
\usepackage[pagewise]{lineno}
\usepackage{amsmath,amsfonts,amssymb,amsthm}
\usepackage{mathtools,mathrsfs,yhmath}
\usepackage{graphicx,color,xcolor}
\usepackage{enumitem}
\usepackage{blindtext}
\usepackage{multicol}
\usepackage{color}
\usepackage{comment}
\usepackage{wrapfig}
\usepackage{dsfont}
\usepackage{graphicx}
\usepackage{thmtools}
\usepackage[backend=bibtex, sorting=nty, doi=false, url=false, isbn=false, maxbibnames=5]{biblatex}

\bibliography{McKeanLandau}

\newcommand\blfootnote[1]{%
	\begingroup
	\renewcommand\thefootnote{}\footnote{#1}%
	\addtocounter{footnote}{-1}%
	\endgroup
}

 \allowdisplaybreaks
\numberwithin{equation}{section}
\numberwithin{equation}{section}
\newtheorem{theorem}{Theorem}[section]
\newtheorem{lemma}[theorem]{Lemma}
\newtheorem{proposition}[theorem]{Proposition}
\newtheorem{corollary}[theorem]{Corollary}

\newtheorem{remark}[theorem]{Remark}
\newcommand\numberthis{\addtocounter{equation}{1}\tag{\theequation}}

\definecolor{orman}{rgb}{0.24, 0.7, 0.44}

\usepackage{hyperref}
\hypersetup{colorlinks=true, linkcolor=blue, filecolor=magenta, urlcolor=blue, citecolor=red}

\usepackage[capitalize, nameinlink]{cleveref}
\crefname{enumi}{}{parts}
\crefname{assumption}{Assumption}{Assumptions}

\def\d{\,\mathrm{d}}

\newcommand{\Id}{\operatorname{Id}}
\newcommand{\R}{\mathbb{R}}

\newcommand{\lp}{\left(}
\newcommand{\rp}{\right)}

\let\eps\varepsilon
\let\mc=\mathcal

\DeclarePairedDelimiter{\abs}{\lvert}{\rvert}
\DeclarePairedDelimiter{\norm}{\lVert}{\rVert}
\makeatletter
\let\oldabs\abs
\def\abs{\@ifstar{\oldabs}{\oldabs*}}
\let\oldnorm\norm
\def\norm{\@ifstar{\oldnorm}{\oldnorm*}}
\makeatother

\hypersetup{pdftitle={McKeanLandau}}
\hypersetup{pdfauthor={Jonathan Junn\'e, Raphael Winter  \& Havva Yoldaş }}

\author{Jonathan Junn\'e\footnote{University of Rennes, Inria, CNRS, IRMAR -- UMR 6625, F-35000 Rennes, France. \href{mailto:jonathan.junne@inria.fr}{jonathan.junne@inria.fr}}
	\and Raphael Winter\footnote{School of Mathematics, Cardiff University,  Senghennydd Road, Cardiff CF24 4AG, UK.  \href{mailto:WinterR6@cardiff.ac.uk}{WinterR6@cardiff.ac.uk}}
	\and Havva Yolda\c{s}\footnote{Delft Institute of Applied Mathematics, Faculty of Electrical Engineering, Mathematics and Computer Science, Delft University of Technology, Mekelweg 4, 2628CD Delft, The Netherlands. \href{mailto:H.Yoldas@tudelft.nl}{H.Yoldas@tudelft.nl}}}
	
\title{A counterexample to McKean's conjecture for the Landau-Coulomb equation}
\date{\today}
\begin{document}
	\maketitle

	\begin{abstract}
	We disprove McKean's conjecture, which asserts that the entropy dissipation is monotone nonincreasing along solutions, or equivalently that the entropy is convex in time, for the spatially homogeneous Landau-Coulomb equation. We provide an explicit counterexample which consists of a Maxwellian equilibrium under radially symmetric, smooth perturbation on an annulus of scale $R\gg 1$. Developing the derivative of the entropy dissipation in powers of $R$ reveals that the leading order terms can be positive for a suitably chosen perturbation. Remarkably, the counterexamples are close to equilibrium in relative entropy and standard Sobolev norms. Since the Landau equation arises from the Boltzmann equation in the grazing collisions limit, our counterexample also shows that McKean's conjecture is not true in general for the Boltzmann equation.
		
	\end{abstract}

	\tableofcontents
	
	\section{Introduction}

The $H$-Theorem of Boltzmann is the fundamental pillar of statistical thermodynamics. For a collisional kinetic equation describing the evolution of a probability distribution $f$ of particles, it asserts that the entropy,
\begin{align*}
	H (f) \coloneq  \int f \log f 
\end{align*}
	is monotone decreasing in time, i.e., 
	\begin{align*}
		D_H(f) \coloneq -\frac{\d}{\d t } H (f) \geq 0.
	\end{align*}
It is also well known that the dissipation $D_H(f)$ vanishes if and only if $f$ is a Maxwellian velocity distribution. 

We introduce the centered Maxwellian distribution $M_T$ with temperature $T>0$ given by
\begin{align}
	\label{def:Maxwellian}
	M_T(v) \coloneq  (2\pi T)^{-3/2}
	\exp\lp -\frac{|v|^2}{2T}\rp, \quad v \in \R^3.
\end{align}

In~\cite{McKean63, McKean66}, McKean conjectured, originally for the one-dimensional Kac's dynamics~\cite{Kac56}, that higher derivatives of the entropy have alternating signs. We prove that this conjecture does not hold for the spatially homogeneous Landau--Coulomb equation~\cite{L58} given by 
 \begin{align}\label{eq:Landau_Coulomb}
 	\partial _t f = Q(f,f) = \nabla_v \cdot \Big(A  [ f  ] \nabla_v f - \lp\nabla_v \cdot A [ f ]\rp f \Big),
 \end{align}
where the diffusion matrix $A [f]$ associated with $f$ is defined as 
\begin{align*}
	A [f] (v) \coloneq  \int_{\R^3} \frac{1}{|v-w|} \Pi_{v-w}^\perp f(w) \d w.
\end{align*}
The Landau collision operator is classically interpreted as the diffusion approximation of the Boltzmann collision operator in a regime dominated by grazing collisions~\cite{AV04, DL-D92, De92, BW23}. 

It is well-known that the Landau operator can also be written in the nondivergence form
\begin{align}\label{eq:non_div}
	 Q(f,f)=A[f]:\nabla^2f+8\pi f^2.
\end{align}

\medskip
	We will prove that the time derivative of the entropy dissipation is positive for perturbation of the unit Maxwellian $ M(v) = M_1(v) $ of the form 
\begin{align} \label{def:fR}
	f_R(v) = M(v) + R^{-K-3} G(v/R) \eqcolon M(v) + g_R(v)
\end{align}
where $K>2$, $R$ large enough and $G$ is the explicit function given by
\begin{align} \label{def:G}
	G (y) \coloneqq \begin{cases}
		\frac{1}{Z} \exp (-|y|^2- |y|^{-2}),  \quad Z \coloneqq \int_{\R^3} \exp (-|y|^2- |y|^{-2}) \d y, &y \neq 0, \\
		0, \quad &y =0.
	\end{cases}
\end{align}
To avoid additional scaling factors, we will work with the functions $f_R$ without normalizing them to probability densities. Of course, all results below also hold for the normalized functions. 

\begin{theorem}[Failure of McKean's conjecture for Landau-Coulomb equation] \label{thm:counter_example} 
Let $f_R$ be the smooth, exponentially decaying sequence of functions given in~\eqref{def:fR} with $R>1$. Further let $F_R$ be the solution of~\eqref{eq:Landau_Coulomb} with initial datum $f_R$. Then, the entropy dissipation increases near $t=0$ for $R$ large enough. More precisely, for $R$ large enough, 
	\begin{align*}
			\left.\frac{\d}{\d t}D_H(F_R(t))\right|_{t=0} >0. 
	\end{align*}
	
\end{theorem}
The proof of the theorem can be found in~\cref{sec:proof_end}.

\begin{remark}[Closeness to equilibrium] We observe that the profile $f_R$, defined in~\eqref{def:fR}, can be chosen arbitrarily close to equilibrium in entropy, Fisher information and Sobolev norms by choosing $K>2, \, R>1$ large enough.
\end{remark}

As a corollary, we obtain that McKean's conjecture is also false for a wide class of kernels for the Boltzmann equation, since the two equations are connected by the grazing collisions limit. We will work in the framework of $H$-solutions introduced by Villani in~\cite{Vi98}. More precisely, consider $H$- solutions $F_{R,\eps}$ with initial datum $f_R$ to a sequence of Boltzmann equations
\begin{align*}
    \partial_t F_{R,\eps} = Q_{B,\eps}(F_{R,\eps},F_{R,\eps}),
\end{align*}
where $(Q_{B,\eps})$ is a sequence of Boltzmann operators 
\begin{align*}
    Q_{B,\eps}(f,f) = \int_{\R^3} \d v_* \int_{S^2} \d \omega B_\eps(v-v_*,\omega) (f'f_*'-ff_*)  ,
\end{align*}
using the standard notation $v'=v-(v-v_*,\omega) \omega$, $v'_*=v_*+(v-v_*,\omega) \omega$ . In the following, we will assume that the sequence of Boltzmann collision operators $(B_\eps(v-v_*,\omega))$ is \emph{grazing} in the sense of~\cite{Vi98}. We also introduce the dissipation of entropy along the Boltzmann equation given by
\begin{align}\label{def:Boltz_Diss}
    D_{H,\eps}(f) = \frac14 \int_{\R^3} \int_{\R^3} \d v \d v_* \int_{S^2} \d \omega B_\eps (v-v_*,\omega) (f'f_*'-ff_*) ( \log(f'f_*')- \log(ff_*)).
\end{align}

\begin{corollary}[Failure of McKean's conjecture for grazing Boltzmann collisions]
\label{cor:Boltzmann}
    Let $R>0$ be large enough and $f_R$ as in~\cref{thm:counter_example}, and $(Q_{B,\eps})$ be a grazing sequence of Boltzmann operators converging to the Landau--Coulomb operator. Then for $\eps>0$ small enough, the entropy dissipation increases at some positive time; there exists a sufficiently small $t^*_\eps$ such that 
    \begin{align*}
        D_{H,\eps}(F_{R,\eps}(t^*_\eps))> D_{H,\eps}(f_R). 
    \end{align*}
\end{corollary}
\begin{proof}
    We use inequality (54) in~\cite{Vi98}, which shows that, for every $t_0 > 0$,
    \begin{align*}
        \int_0^{t_0} D_H(F_R(\tau)) \d \tau \leq \liminf_{\eps \rightarrow 0} \int_0^{t_0} D_{H,\eps} (F_{R,\eps}) \d \tau.
    \end{align*}
    For $t_0>0$ small enough,~\cref{thm:counter_example} shows 
    \begin{align*}
        t_0 D_H(f_{R}) < \int_0^{t_0} D_H(F_R(\tau)) \d \tau.
    \end{align*}
    Since $\lim_{\eps \rightarrow 0} D_{H,\eps} (f_R) = D_{H}(f_R)$, we infer that for $\eps>0$ small enough,
    \begin{align*}
        t_0 D_{H,\eps}(f_R) < \int_0^{t_0} D_{H,\eps} (F_{R,\eps}(\tau)) \d \tau,
    \end{align*}
    and therefore there exists $t^*_\eps \in(0,t_0)$ such that
    \begin{align*}
       D_{H,\eps}(f_R) < D_{H,\eps} (F_{R,\eps}(t^*_\eps))
    \end{align*}
    as claimed.    
\end{proof}

\paragraph*{State of the art.} In~\cite{McKean66}, McKean further conjectured that all the successive time derivatives of the entropy alternate in signs, i.e.,
\begin{align*}
	(-1)^n \frac{\d^n}{\d t^n} H (f) \geq 0, \quad n \geq 1,
\end{align*} 
where $f$ solves one-dimensional Kac's caricature of the spatially homogeneous Boltzmann equation. This stronger version of the conjecture is sometimes referred to as \emph{super-H conjecture}~\cite{Ol82}. In~\cite{Ha67}, Harris proved the super-H conjecture for a toy model with discrete velocities. Bobylev~\cite{Bo75} and Krook--Wu~\cite{KW76, KW77} constructed a particular form of explicit, non-equilibrium solutions, called \emph{BKW solutions} to the nonlinear Boltzmann equation with Maxwell molecules. Then, Olaussen~\cite{Ol82} and Lieb~\cite{Lieb82} disproved the super H-conjecture by studying the BKW solutions. Particularly, in~\cite{Ol82} it was shown that the sign fails starting from $n\geq 102$, whereas the abstract argument in~\cite{Lieb82} excludes the complete monotonicity without giving an explicit $n$ where it starts to fail. Moreover, numerical simulations in~\cite{ZMS81} indicate the complete monotonicity of $H$ for $n\leq 30$ for the BKW solutions. All these results shows that the full super H-conjecture is false but for small values of $n$ it might still hold. 

In~\cite{CG15}, the authors prove that the third and fourth derivatives of entropy along the Gaussian heat flow have alternating signs, thereby confirming McKean’s conjecture up to fourth order, while leaving the full all-orders conjecture open. In a recent preprint~\cite{GS26}, the authors provided a probability measure for which the fifth time derivative of the entropy along the heat flow is positive at some time. Moreover, recently in~\cite{Si26}, Silvestre showed that the entropy production is not always monotone in the case of spatially homogeneous Boltzmann equation with a non-physical collision kernel, and for hard spheres and Maxwell molecules in~\cite{Si26-2}. We also refer to~\cite{Si26, Vi25} for more detailed literature references on McKean's conjecture. 
Silvestre's work raises the  question whether monotonicity of the entropy production is true or false in physically relevant cases.  A corollary of our main result on the Landau--Coulomb equation answers this question by disproving McKean's conjecture for the spatially homogeneous Boltzmann equation in the settings where it leads to the Landau equation in the gazing collisions limit. 

Another related question concerns the monotonicity of the Fisher information which was answered positively in recent breakthrough results by Guillen--Silvestre in~\cite{GS26} for the Landau equation and by Imbert--Silvestre--Villani in~\cite{ISV25} for the Boltzmann equation. Later, we showed that this is not true in general for the multi-species Landau equation in~\cite{JWY25}.

\section{Failure of McKean's conjecture for the sequence $f_R$} \label{sec:proof}
 
 In this section, we prove that the sequence $f_R$ defined in~\eqref{def:fR} is indeed a valid a counterexample for the monotonicity of the entropy dissipation. To this end, we first compute the time derivative of the entropy dissipation.
 
\begin{proposition}[Derivative of entropy dissipation]
	\label{prop:exact-derivative-identity}
	Let $f$ be a classical solution of the Landau-Coulomb equation~\eqref{eq:Landau_Coulomb}. Then
	\begin{equation}\label{eq:exact-derivative-Q-form}
		\frac{\d}{\d t}D_H(f)
		=\int_{\R^3}Q(f,f)\,\delta_{f}\d v	-2\int_{\R^3}\frac{Q(f,f)^2}{f}\d v	,
	\end{equation}
where $\delta_f$ is given by
\begin{align}\label{def:delta_f}
	 \delta_f(v) \coloneqq 	\int_{\R^3} 	f_* \frac{	\left| 	\Pi_{v-v_*} ^\perp	\bigl(\nabla\log f-\nabla_*\log f_*\bigr)\right|^2}{|v-v_*|} 	\d v_*.
\end{align}
\end{proposition}

Our goal is to prove that for $f_R$ given by~\eqref{def:fR} the right hand side of~\eqref{eq:exact-derivative-Q-form} is positive. We therefore introduce
\begin{align} \label{def:I_II_R}
		\frac{\d}{\d t}D_H(F_R) =
	\int_{\R^3}Q(f_R,f_R)\,\delta_{f_R}\d v	-2\int_{\R^3}\frac{Q(f_R,f_R)^2}{f_R}\d v = \mbox{I}_R+  \mbox{II}_R .
\end{align} 
The term $\mbox{I}_R$ will contain the leading order term, whereas the term $\mbox{II}_R$ will be an error term. 

To this end, we need to expand both $Q(f_R,f_R)$ and $\delta_{f_R}$. We  observe that, since $M$ is an equilibrium distribution, we have
\begin{align}\label{Pi_identity}
	\Pi_{v-v_*}^\perp	\lp \nabla\log f_R-\nabla_*\log f_{R*}\rp =  \Pi_{v-v_*}^\perp \lp \frac{v g_R+ \nabla g_R}{M + g_R} - \frac{v_* g_{R*}+ \nabla_* g_{R*}}{M_* + g_{R*}} \rp. 
\end{align}

For a more detailed expansion, we decompose the terms $\mbox{I}_R$ and $\mbox{II}_R$ further into the contributions of the annuli 
\begin{align}\label{domain}
	\Omega_R^{\rm in} \coloneqq \{|v| \leq R^\frac25\}, \qquad	\Omega_R^{\rm mid} \coloneqq \{R^\frac25 \leq |v| \leq R^\frac35\}, \qquad 
	\Omega_R^{\rm out}  \coloneqq \{|v| \geq R^\frac35\}.
\end{align}
This leads to the decomposition of   $\mbox{I}_ R, \mbox{II}_R$ defined in~\eqref{def:I_II_R} into, writing $Q \coloneq Q(f_R,f_R)$,
\begin{align*}
	\mbox{I}_R &= \int_{\Omega_R^{\rm in}}Q\,\delta_{f_R}\d v	 + \int_{\Omega_R^{\rm mid}}Q\,\delta_{f_R}\d v + \int_{\Omega_R^{\rm out}}Q\,\delta_{f_R}\d v \coloneq \mbox{I}_R ^{\rm in} + \mbox{I}_R ^{\rm mid} +  \mbox{I}_R ^{\rm out}, \\
		\mbox{II}_R &= -2\int_{\Omega_R^{\rm in}}\frac{Q^2}{f_R}\d v - 2\int_{\Omega_R^{\rm mid}}\frac{Q^2}{f_R}\d v- 2\int_{\Omega_R^{\rm out}}\frac{Q^2}{f_R}\d v \coloneq \mbox{II}_R ^{\rm in} + \mbox{II}_R ^{\rm mid} +  \mbox{II}_R ^{\rm out}.
\end{align*}

We will show that $\mbox{I}_R ^{\rm out}$ gives the positive, leading order contribution of order $O(R^{-K-4})$, and the remaining terms are of higher order. This is the content of the remaining sections. Throughout the proofs, generic constants $c,C > 0$ may change from line to line and within the same displayed equation.

\subsection{Estimates of the inner region contributions}

\begin{proposition}[Inner contribution]
	\label{prop:inner-contribution}
	There exists $C>0$ such that, for all sufficiently large $R$,
	\begin{equation}\label{prop:inner_estimate}
	|\RN{1}_R ^{\rm in}| +	|\RN{2}_R ^{\rm in}| \leq 	2\int_{\Omega_R^{\rm in}} \lp
		\frac{Q(f_R,f_R)^2}{f_R}	+	|Q(f_R,f_R)|\,\delta_{f_R} 	\rp\d v \leq	CR^{-2K-2}.
	\end{equation}
\end{proposition}

\begin{proof}
	We will make use of the elementary inequality 
	\begin{align}\label{g_R_inner_der}
		|g_R(v)| +  |\nabla g_R(v)|  +|\nabla^2 g_R(v)| \leq C M(v) \exp \lp -cR^{\frac65} \rp \quad  \text{for all } v \in \Omega_R^{\rm in}.
	\end{align}
Using the non-divergence form~\eqref{eq:non_div} of the operator $Q(f,f)$ and the fact that $Q(M, M) =0$, we obtain 
\begin{align}\label{Q_f_R_expansion}
	Q(f_R, f_R) = A[g_R] : \nabla^2 M + A[M] : \nabla^2 g_R + 16 \pi M g_R + Q(g_R, g_R).
\end{align}
Now $A[g_R]$ can be estimated by 
\begin{align*}
|A[g_R] (v)| \leq C R^{-K-1}.
\end{align*}
Since the other terms are exponentially small we can estimate 
\begin{align}\label{Q_R_estimate}
	 |Q(f_R,f_R)(v)|
	\leq
	CR^{-K-1}\langle v\rangle^2M(v) \qquad \text{for all } v \in \Omega_R^{\rm in}.
\end{align}
Using this we can already estimate the first term in~\eqref{prop:inner_estimate}
\begin{align*}
	\int_{\Omega_R^{\rm in}} 	\frac{Q(f_R,f_R)^2}{f_R} \d v \leq C^2	\int_{\R^3}  R^{-2K-2}\langle v\rangle^4M(v ) \d v \leq C R^{-2K-2}.
\end{align*}
It remains to estimate the contriution of $\delta_{f_R}$. 
We combine the identity~\eqref{Pi_identity} with~\eqref{g_R_inner_der} to obtain
\begin{align*}
\left 	|\Pi_{v-v_*}^\perp	\lp \nabla\log f_R-\nabla_*\log f_{R*}\rp \right  | \leq C   \langle v\rangle \exp \lp -c R^{\frac65}\rp + C \lp |v| + \frac{|v_*|}{R^2} + \frac{R^2}{|v_*|^3} \rp \frac{g_{R*}}{M_*+ g_{R*}} \quad \text{for all } v \in \Omega_R^{\rm in}.
\end{align*}
	This allows us to estimate $\delta_{f_R}$ for all $v \in \Omega_R^{\rm in}$ by 
	\begin{align*}
		\delta_{f_R} (v) &=  \int_{\R^3} 	f_{R*} \frac{	\left| 	\Pi_{v-v_*} ^\perp	\bigl(\nabla\log f_R-\nabla_*\log f_{R*}\bigr)\right|^2}{|v-v_*|} 	\d v_* \\
		&\leq C \langle v\rangle^2 \exp \lp -2c R^{\frac65}\rp + C \int_{\R^3}  \lp |v|^2 + \frac{|v_*|^{2}}{R^4} + \frac{R^4}{|v_*|^6} \rp\frac{g_{R*}^2}{M_*+ g_{R*}} |v-v_*|^{-1}\d v_*
		\\& \leq C \langle v\rangle^2 \exp \lp -2c R^{\frac65}\rp + C R^{-1}\int_{\R^3} \lp |v|^2 +  \lp \left|\frac{v_*}{R} \right |^2 + \left | \frac{v_*}{R} \right |^{-6} \rp  R^{-2} \rp g_{R*} \left |\frac{v-v_*}{R} \right |^{-1}\d v_*
		\\ &\leq C R^{-K-1} \langle v\rangle^2,
	\end{align*} 
where we have used that the profile $G$ is flat near the origin.

Inserting this estimate and~\eqref{Q_R_estimate} back into~\eqref{prop:inner_estimate} finishes the proof. 
	
\end{proof}

\subsection{Estimates of the middle region contributions}

\begin{proposition}\label{prop:middle-contribution}
	There exist $c,C>0$ such that, for all sufficiently large $R$,
	\begin{equation} \label{prop:mid_estimate}
		\int_{\Omega_R^{\rm mid}}	\lp	\frac{Q(f_R,f_R)^2}{f_R}	+|Q(f_R,f_R)|\,\delta_{f_R}	\rp \d v	\leq	C\exp \lp -cR^\frac45 \rp .
	\end{equation}
\end{proposition}
\begin{proof}
	On the middle region $\Omega_R^{\rm mid} =\{R^\frac25 \leq |v| \leq R^\frac35\}$, $f_R$ and $g_R$ satisfy the estimates
	\begin{align} \label{mid_f_R_g_R}
		f_R + g_R \leq C \exp \lp -c R^\frac45\rp, \qquad |\nabla g_R| \leq C R^Ng_R, \qquad |\nabla^2 f_R| \leq C R^N f_R, \quad \text{for some }N>0.
	\end{align}
The non-divergence form~\eqref{eq:non_div} therefore gives
\begin{align}\label{eq:mid-Q-bound}
	|Q(f_R,f_R)(v)|	\leq	CR^Nf_R	(v)\qquad	\text{for all } v \in \Omega_R^{\rm mid}.
\end{align}
Recalling the identity~\eqref{Pi_identity} allows us to find the bound
\begin{align*}
\left 	|	\Pi_{v-v_*}^\perp	\lp \nabla\log f_R-\nabla_*\log f_{R*}\rp \right | \leq \left  | \Pi_{v-v_*}^\perp \lp \frac{v g_R+ \nabla g_R}{M + g_R} - \frac{v_* g_{R*} +\nabla_* g_{R*}}{M_* + g_{R*}} \rp \right| \leq C \lp |v| + R^N + |v_*|\rp,
\end{align*}
which in turn shows that $\delta_{f_R}$ is bounded by, after possibly increasing $N$,
\begin{align} \label{delta_mid_est}
	\delta_{f_R} = \int_{\R^3} 	f_{R*} \frac{	\left| 	\Pi_{v-v_*} ^\perp	\bigl(\nabla\log f_R-\nabla_*\log f_{R*}\bigr)\right|^2}{|v-v_*|} 	\d v_* \leq CR^N \quad \text{for all }v \in \Omega_R^{\rm mid}.
\end{align}
	Inserting~\eqref{mid_f_R_g_R}, \eqref{eq:mid-Q-bound} and~\eqref{delta_mid_est} into~\eqref{prop:mid_estimate} gives 
	\begin{align*}
			\int_{\Omega_R^{\rm mid}}	\lp	\frac{Q(f_R,f_R)^2}{f_R}	+|Q(f_R,f_R)|\,\delta_{f_R}	\rp \d v \leq 	C \int_{\Omega_R^{\rm mid}}	 R^{2N} f_R \d v \leq C \int_{\Omega_R^{\rm mid}}	 R^{2N} \exp \lp -c R^{\frac45}\rp\d v. 
	\end{align*}
Absorbing polynomial factors in the choices of $C, c>0$, we obtain the claim.
\end{proof}

\subsection{Estimates of the outer region contributions}

\begin{proposition}[Expansion of the collision operator]
	\label{prop:exp_collision}
	The collision operator has the expansion
	\begin{equation}\label{eq:outer_Q_exp}
		Q(f_R,f_R)(v)	=
		\bigl(Q_{R,1}+Q_{R,2}+Q_{R,3}\bigr)(v/R)	+
		\mathcal E_R^Q(v) \qquad\text{for all }v \in \Omega_R^{\rm out},
	\end{equation}
	where
	\begin{equation}\label{eq:out_leading_terms}
		Q_{R,1}	=R^{-K-6}\frac{\Pi^\perp_v}{|v|}:\nabla^2G,
		\qquad	Q_{R,2}	=\frac{R^{-K-8}}{|v|^3}
		\lp3\frac{v\otimes v}{|v|^2}-\Id	\rp:\nabla^2G,
		\qquad	Q_{R,3} 	=R^{-2K-6}Q(G,G).
	\end{equation}
	For every $N\geq0$, there exist constants $c_N,C_N>0$ such that, for all
	sufficiently large $R$,
	\begin{equation}\label{eq:out_error_est}
		\int_{\Omega_R^{\rm out}}\left(
		\frac{|\mathcal E_R^Q(v)|^2}{g_R(v)}
		+	|\mathcal E_R^Q(v)|\langle v\rangle^N	\right)\d v
		\leq
		C_N \exp \lp -c_NR^{6/5} \rp.
	\end{equation}
\end{proposition}

\begin{proof}
	By \eqref{eq:core-diffusion-expansion}, the diffusion matrix associated to the Maxwellian $M$ has the expansion
	\begin{align*}
		A[M](v) 	=
		\frac{1}{|v|}\Pi^\perp_v +\frac{1}{|v|^3}	\left(	3\frac{v\otimes v}{|v|^2}-\Id	\right) +	\mathcal E_A(v)
		\qquad \text{for $|v |>1$},
	\end{align*}
where $\mc E_A (v) \leq C \exp \lp -c|v|^2 \rp$. 
	Since $Q(M,M)=0$, the non-divergence form~\eqref{eq:non_div} gives
	\begin{align*}
			Q(f_R,f_R)	= A[M]:\nabla^2g_R +Q(g_R,g_R) +A[g_R]:\nabla^2M
		+16\pi Mg_R.
	\end{align*}
	Using the notation $y=v/R$, we can rewrite
	\begin{align*}
		\nabla^2g_R(v)	= R^{-K-5}\nabla^2G(y),	\qquad		Q(g_R,g_R)(v) =	R^{-2K-6}Q(G,G)(y).
	\end{align*}
Consequently,
\begin{align*}
	Q(f_R,f_R)(v)	=	\bigl(Q_{R,1}+Q_{R,2}+Q_{R,3}\bigr)(v/R)	+	\mathcal E_R^Q(v),
\end{align*}
where
	\begin{equation}\label{eq:explicit-Q-remainder}
		\mathcal E_R^Q	=
		\mathcal E_A:\nabla^2g_R	+A[g_R]:\nabla^2M+16\pi Mg_R.
	\end{equation}
	This proves the expansion.
	
	It remains to establish \eqref{eq:out_error_est}.  Computing explicit derivatives gives, for all $v\in \Omega_R^{\rm out}$,
	\begin{align*}
		M(v)\leq	Cg_R(v)\exp \lp -c|v|^2/R^2 \rp,
		\quad |\nabla^2g_R(v)|	\leq C\langle v\rangle^2g_R(v),
		\quad	|\nabla^2M(v)|	\leq
		C\langle v\rangle^2M(v).
	\end{align*}
	Using this and the uniform boundedness of $A[g_R]$ we get 
\begin{align*}
	\frac{|\mathcal E_A:\nabla^2g_R|^2}{g_R}
	&\leq
	C\langle v\rangle^4 \exp \lp -2c|v|^2 \rp g_R,\\
	\frac{|A[g_R]:\nabla^2M|^2}{g_R}
	&\leq
	C\langle v\rangle^4\frac{M^2}{g_R}
	\leq
	C\langle v\rangle^4M(v) \exp \lp -c|v|^2/R^2 \rp,\\
	\frac{|16\pi Mg_R|^2}{g_R}
	&\leq
	CM^2g_R.
\end{align*}
Taking the integrals of the the inequalities above yields
\begin{align*}
	\int_{\Omega_R^{\rm out}}	\frac{|\mathcal E_R^Q|^2}{g_R}\d v\leq	C \exp \lp -cR^\frac65 \rp, 
\end{align*}
and the Cauchy-Schwarz inequality reveals
	\begin{align*}
		\int_{\Omega_R^{\rm out}}
		|\mathcal E_R^Q(v)|\langle v\rangle^N\d v
		\leq
		\left(
		\int_{\Omega_R^{\rm out}}
		\frac{|\mathcal E_R^Q|^2}{g_R}\d v
		\right)^{1/2} 
		\left(
		\int_{\R^3}g_R(v)\langle v\rangle^{2N}\d v
		\right)^{1/2}.
	\end{align*}
	The second factor grows at most polynomially in $R$, and is therefore
	absorbed by the stretched exponential decay, upon possibly decreasing $c > 0$.  This proves~\eqref{eq:out_error_est}.
\end{proof}

\begin{proposition}
	\label{prop:out_delta_exp} The functional $\delta_{f_R}$ for $f_R$ defined in~\eqref{def:fR}  has the expansion
	\begin{equation}\label{eq:delta_exp}
		\delta_{f_R}(v)	=\lp \delta_{R,1} +\delta_{R,2} +\delta_{R,3} +  \delta_{R,4}\rp (v/R)  
		+	\mc E_R^\delta(v) \qquad \text{for all $ v \in \Omega_R^{\rm out}$},
	\end{equation}
where
\begin{align}\label{eq:out_delta_terms}
		\delta_{R,1}	=\frac{2}{R|\cdot|},
		\qquad	\delta_{R,2}	=\frac{8}{R^3|\cdot|}\beta,
		\qquad	\delta_{R,3} 	=\frac{8}{R^5|\cdot|}\beta^2, \qquad \delta_{R,4} = R^{-K-3}\delta_G ,
\end{align}
and $\beta (y) \coloneq |y|^{-4} -1$.

For every $N\geq0$, there exist constants $c,C_N>0$ such that, for all sufficiently large $R$,
	\begin{equation}\label{delta_gR_error}
		\int_{\Omega_R^{\rm out}}	g_R(v)\langle v \rangle^N
		\lp	|\mathcal E_R^\delta(v)|	+	|\mathcal E_R^\delta(v)|^2 \rp \d v
		\leq
		C_N \exp \lp -cR^{\frac45} \rp.
	\end{equation}
\end{proposition}
\begin{proof}
		We split $\delta_{f_R}$ as
	\begin{align*}
		\delta_{f_R} &\coloneqq	\int_{\R^3}M_*
		\frac{\left|	\Pi^\perp_{v-v_*}	\lp \nabla\log f_R-\nabla_*\log f_{R*}\rp	\right|^2	}{|v-v_*|}	\d v_* + 	\int_{\R^3}g_{R*}
		\frac{\left|	\Pi^\perp_{v-v_*}	\lp \nabla\log f_R-\nabla_*\log f_{R*}\rp	\right|^2	}{|v-v_*|}	\d v_* \\ &\eqcolon\delta_{f_R}^M + \delta_{f_R}^{g_R}.
	\end{align*}
\textbf{Step 1: Expansion of $\delta_{f_R}^M $.} We perform an asymmetric expansion of the logarithms of $f_R$,
\begin{align*}
\nabla_*  \log f_{R*}  = \nabla_*\log M_* + \eta_{R*}^M , \qquad 	\nabla  \log f_R  = \nabla \log  g_R + \eta_R^{g_R},
\end{align*}
where $\eta_{R*}^M , \eta_R^{g_R}$ are error functions whose contribution we estimate below. 
Using this, we expand $\delta_{f_R}^M$,
\begin{align} \label{delta_M_exp}
	\delta_{f_R}^M = & \int_{\R^3}M_*\frac{	\left| \Pi_{v-v_*}^\perp\lp 	\nabla\log g_R	-\nabla_*\log M_*	\rp	\right|^2}{|v-v_*|}\d v_* \notag \\
	+	&2\int_{\R^3}M_*\frac{\Pi_{v-v_*}^\perp \lp	\nabla\log g_R	-\nabla_* \log M_* \rp}{|v-v_*|}
	\cdot \Pi_{v-v_*}^\perp	\lp
	\eta_R^{g_R}-\eta_{R*}^M \rp	\d v_*\\
	+&	\int_{\R^3}M_*	\frac{	\left|	\Pi_{v-v_*}^\perp\lp
		\eta_R^{g_R}-\eta_{R*}^M \rp	\right|^2	}{|v-v_*|}	\d v_*. \notag
\end{align}
For the first term we observe that 
\begin{align} \label{g_RM*_identity1}
	\Pi_{v-v_*}^\perp \lp \nabla\log g_R	-\nabla_*\log M_* \rp
	=\lp 	1	+	\frac{2\beta(v/R)}{R^2} \rp 	\Pi_{v-v_*}^\perp v.
\end{align}
Hence the first integral in~\eqref{delta_M_exp} is equal to 
\begin{align*}
	\int_{\R^3}M_*\frac{	\left| \Pi_{v-v_*}^\perp\lp 	\nabla\log g_R	-\nabla_*\log M_*	\rp	\right|^2}{|v-v_*|}\d v_* = \lp 	1	+	\frac{2\beta(v/R)}{R^2} \rp^2	A[M] : v \otimes v.
\end{align*}
We recall that $A[M]$ can be expanded as
\begin{align*}
	A[M](v) 	=
\frac{1}{|v|}\Pi^\perp_v +\frac{1}{|v|^3}	\left(	3\frac{v\otimes v}{|v|^2}-\Id	\right) +	\mathcal E_A(v)
\qquad \text{for $|v |>1$},
\end{align*}
In conclusion, we obtain
\begin{align}\label{eq:decomposition-delta-fM-with-error}
	\delta_{f_R}^M  = \lp \frac{2}{R|\cdot|} +\frac{8}{R^3|\cdot|}\beta +\frac{8}{R^5|\cdot|}\beta^2\rp (v/R) + \lp 	1	+	\frac{2\beta(v/R)}{R^2} \rp^2 \mc E_A(v) : v \otimes v + \mathcal R^M_R(v),
\end{align}
where
\begin{align*}
	\mathcal R_R^M(v)
	\coloneqq{}
	2&\int_{\R^3}M_*
	\frac{
		\Pi^\perp_{v-v_*}
		\lp \nabla\log g_R-\nabla_*\log M_* \rp
	}{|v-v_*|}
	\cdot
	\Pi^\perp_{v-v_*}
	\lp \eta_R^{g_R}-\eta_{R*}^M \rp
	\d v_*
	\\
	+
	&\int_{\R^3}M_*
	\frac{
		\left|
		\Pi^\perp_{v-v_*}
		\lp \eta_R^{g_R}-\eta_{R*}^M \rp
		\right|^2
	}{|v-v_*|}
	\,\d v_* .
\end{align*}

Since $|v/R| \geq R^{-\frac25}$ in the outer region, we know that 
\begin{align}\label{g_RM*_identity2}
	1	+	\frac{2\beta(v/R)}{R^2}  \leq 1 + \frac{2 (R^\frac85 -1)}{R^2} \leq C
\end{align}
is bounded. Hence the second term of~\eqref{eq:decomposition-delta-fM-with-error} satisfies the error bound~\eqref{delta_gR_error} recalling that $\mc E_A (v) \leq C \exp \lp -c|v|^2 \rp$. 
To continue further, we first observe that 
\begin{align*}
	\eta_R^{g_R} = \frac{M}{f_R} \lp \nabla \log M - \nabla \log g_R\rp, \qquad \eta_{R_*}^{M} = \frac{g_{R*}}{f_{R*}} \lp \nabla \log g_{R*} - \nabla_* \log M_*\rp.
\end{align*}
We can then use this to estimate, for $v \in \Omega_R^{\rm out}$, 
\begin{align*}
	|	\eta_R^{g_R}| \leq C  \exp \lp -c|v|^2 \rp,
\end{align*}
while for the second error function
\[M_*|\eta_{R*}^M|^2 \le C \min\{M_*, g_{R*}\}(|v_*|^2 + R^4|v_*|^{-6}),\]
which yields, for $v_* \in \Omega_R^{\rm out}$, using the Maxwellian bound,
\begin{align*}
	M_*|\eta_{R*}^M|^2
	\leq
	C\exp\lp-\frac{|v_*|^2}{2}\rp
	\lp |v_*|^2+R^4|v_*|^{-6}\rp
	\leq C\exp(-c|v_*|^2).
\end{align*}

We claim that the third term of~\eqref{delta_gR_exp} satisfies the error bound in~\eqref{delta_gR_error}.
It can be bounded by
\begin{align}\label{error_step1}
	|\mathcal R_{R}^M(v)|
	\leq{}&
	\int_{\R^3} M_*
	\frac{
		\left|
		\Pi_{v-v_*}^\perp
		\lp \eta_R^{g_R}-\eta_{R*}^M \rp
		\right|^2
	}{|v-v_*|}
	\d v_*
	\notag\\
	&+
	2\lp
	\int_{\R^3} M_*
	\frac{
		\left|
		\Pi_{v-v_*}^\perp
		\lp \eta_R^{g_R}-\eta_{R*}^M \rp
		\right|^2
	}{|v-v_*|}
	\d v_*
	\rp^{\frac12}
	\notag\\
	&\hspace{1cm}\times
	\lp
	\int_{\R^3}M_*
	\frac{
		\left|
		\Pi_{v-v_*}^\perp
		\lp \nabla\log g_R-\nabla_*\log M_* \rp
		\right|^2
	}{|v-v_*|}
	\d v_*
	\rp^{\frac12}.
\end{align}
We obtain, for $v\in\Omega_R^{\rm out}$,
\begin{align*}
	\int_{\R^3}M_*
	&\frac{
		\left|
		\Pi_{v-v_*}^\perp
		\lp \eta_R^{g_R}-\eta_{R*}^M\rp
		\right|^2
	}{|v-v_*|}
	\d v_* \leq
	2|\eta_R^{g_R}(v)|^2
	\int_{\R^3}\frac{M_*}{|v-v_*|}\d v_*
	+
	2\int_{\R^3}
	\frac{M_*|\eta_{R*}^M|^2}{|v-v_*|}\d v_*
	\\
    & \leq C\exp\lp-cR^\frac{6}{5}\rp 
    + 
    2\int_{\left\{|v_*| < R^\frac35\right\}} \frac{M_*|\eta_{R*}^M|^2}{|v-v_*|}\d v_* \\
    & \leq C\exp\lp-cR^\frac{6}{5}\rp  
	+
	CR^{-K-2}\int_{\left\{|v_*| < R^\frac35\right\}} \frac{G(v_*/R)}{|v/R-v_*/R|}\left(\left|\frac{v_*}{R}\right|^2 + \left|\frac{v_*}{R}\right|^{-6}\right) \d v_* \\
    & \leq C\exp\lp-cR^\frac{6}{5}\rp +  C R^{-K+1}  \exp \lp -\frac12 R^\frac45 \rp \int_{ \left \{|v_*| < R^{-\frac25} \right \}}  \frac{ e^{-\frac12 |v_*|^{-2}}}{|v/R-v_*|}  \lp |v_*|^2 + |v_*|^{-6}\rp \d v_* \\
	& \leq
	C\exp\lp-cR^\frac45\rp,
	\numberthis \label{error_step11}
\end{align*}
and using~\eqref{g_RM*_identity1} together with~\eqref{g_RM*_identity2} reveals
\begin{align*}
	\int_{\R^3}M_*
	\frac{
		\left|
		\Pi_{v-v_*}^\perp
		\lp \nabla\log g_R-\nabla_*\log M_* \rp
		\right|^2
	}{|v-v_*|}
	\d v_* 
    \le C\int_{\R^3} \frac{M_*|v_*|^2}{|v-v_*|} \d v_* \le C. \numberthis\label{error_step12}
\end{align*}
Inserting the estimates \eqref{error_step11} and \eqref{error_step12} into \eqref{error_step1} proves the claim.

\medskip 
\noindent
\textbf{Step 2: Expansion of $\delta_{f_R}^{g_R} $.} 
This time we expand the logarithms of $f_R$ symmetrically,
\begin{align*}
	\nabla_*  \log f_{R*}  = \nabla_*\log g_{R*} + \eta_{R*}^{g_R}, \qquad 	\nabla  \log f_R  = \nabla \log  g_R + \eta_R^{g_R},
\end{align*} 
and use these to expand $\delta_{f_R}^{g_R} $
\begin{align}	\label{delta_gR_exp}
	\delta_{f_R}^{g_R} = & \int_{\R^3}g_{R*}\frac{	\left| \Pi_{v-v_*}^\perp\lp 	\nabla\log g_R	-\nabla_*\log g_{R*}	\rp	\right|^2}{|v-v_*|}\d v_* \notag \\
	+	&2\int_{\R^3}g_{R*}\frac{\Pi_{v-v_*}^\perp \lp	\nabla\log g_R	-\nabla_* \log g_{R*} \rp}{|v-v_*|}
	\cdot \Pi_{v-v_*}^\perp	\lp
	\eta_R^{g_R}-\eta_{R*}^{g_R} \rp	\d v_*\\
	+&	\int_{\R^3} g_{R*}	\frac{	\left|	\Pi_{v-v_*}^\perp\lp
		\eta_R^{g_R}-\eta_{R*}^{g_R}  \rp	\right|^2	}{|v-v_*|}	\d v_*. \notag
\end{align}
The first integral in~\eqref{delta_gR_exp} can be simplified as 
\begin{align*}
	\int_{\R^3}g_{R*}\frac{	\left| \Pi_{v-v_*}^\perp\lp 	\nabla\log g_R	-\nabla_*\log g_{R*}	\rp	\right|^2}{|v-v_*|}\d v_* = R^{-K-3} \delta_G (v/R).
\end{align*}
It remains to show that two remaining error terms in~\eqref{delta_gR_exp} satisfy the error bound~\eqref{delta_gR_error}. They can be bounded by
\begin{align}\label{error_step2}
			2 \int_{\R^3} \left | g_{R*}\frac{\Pi_{v-v_*}^\perp \lp	\nabla\log g_R	-\nabla_* \log g_{R*} \rp}{|v-v_*|}
	\cdot \Pi_{v-v_*}^\perp	\lp
	\eta_R^{g_R}-\eta_{R*}^{g_R} \rp \right |	\d v_*
	+	\int_{\R^3}  g_{R*}	\frac{	\left|	\Pi_{v-v_*}^\perp\lp
		\eta_R^{g_R}-\eta_{R*}^{g_R}  \rp	\right|^2	}{|v-v_*|}	\d v_* \notag\\
	\leq 	
	\int_{\R^3}   g_{R*}	\frac{	\left|	\Pi_{v-v_*}^\perp\lp
		\eta_R^{g_R}-\eta_{R*}^{g_R}  \rp	\right|^2	}{|v-v_*|}\d v_* +   2 \lp \int_{\R^3} g_{R*}	\frac{	\left|	\Pi_{v-v_*}^\perp\lp
		\eta_R^{g_R}-\eta_{R*}^{g_R}  \rp	\right|^2	}{|v-v_*|}	\d v_*\rp^\frac12 \lp R^{-K-3} \delta_G (v/R)\rp^\frac12.
\end{align}
We obtain, for $v \in \Omega_R^{\rm out}$, 
\begin{align}
	\int_{\R^3}   g_{R*}	& \frac{	\left|	\Pi_{v-v_*}^\perp\lp
		\eta_R^{g_R}-\eta_{R*}^{g_R}  \rp	\right|^2	}{|v-v_*|} \d v_* \leq 2R^{-K-3}\int_{\R^3}   G(v_*/R)  \frac{\lp 
		\eta_R^{g_R} \rp^2+ \lp \eta_{R*}^{g_R}  \rp^2}{|v-v_*|}\d v_* \notag
	\\ &\leq C  \exp \lp -cR^\frac65 \rp	\int_{\R^3}  \frac{ G(v_*/R) }{|v-v_*|} \d v_* + \int_{ \left \{|v_*| < R^\frac35 \right \}}  G(v_*/R)  \frac{\lp \eta_{R*}^{g_R}  \rp^2}{|v-v_*|} \d v_*  \notag
	\\ &\leq C \exp \lp -cR^\frac65 \rp  + C R^{-1} 	\int_{ \left \{|v_*| < R^\frac35 \right \}}  \frac{ G(v_*/R) }{|v/R-v_*/R|}   \lp   \left | \frac{v_*}{R}\right |^2 +  \left | \frac{v_*}{R}\right |^{-6}\rp  \d v_* \notag
	\\& \leq  C \exp \lp -cR^\frac65 \rp  + C R^{2}  \exp \lp -\frac12 R^\frac45 \rp \int_{ \left \{|v_*| < R^{-\frac25} \right \}}  \frac{ e^{-\frac12 |v_*|^{-2}}}{|v/R-v_*|}  \lp |v_*|^2 + |v_*|^{-6}\rp   \d v_* \notag
	\\ &\leq C \exp \lp -cR^\frac45 \rp.\label{error_step21}
\end{align}
By a similar computation as above, one gets
\begin{align}\label{delta_G_est}
	\delta_G (y) \leq C \lp \langle y \rangle^2 + |y|^{-6}\rp.
\end{align}

Inserting the estimates~\eqref{error_step21} and~\eqref{delta_G_est} into~\eqref{error_step2} allows us to conclude the error terms satisfy the bound~\eqref{delta_gR_error}.
\end{proof}

\begin{proposition}[Leading order contribution]  \label{prop:out-contribution}
	For all sufficiently large $R$,
	\begin{equation} \label{prop:out_estimate}
		\int_{\Omega_R^{\rm out}}	\lp	-2\frac{Q(f_R,f_R)^2}{f_R}	+ Q(f_R,f_R)\,\delta_{f_R}	\rp \d v	= \frac83 R^{-K-4}  + O (R^{-K-6}).
	\end{equation}
\end{proposition}
\begin{proof}
	The leading order contribution is given by $Q(f_R, f_R) \delta_{f_R}$. We begin with the error bound for the other term.
	
	\medskip
	\noindent
	\textbf{Step 1: Contribution of $Q(f_R,f_R)^2$.}  We use the decomposition in~\cref{prop:exp_collision} to obtain
	\begin{align*}
		\int_{\Omega_R^{\rm out}}	\frac{Q(f_R,f_R)^2}{f_R}\d v \leq  	C\sum_{i=1}^3\int_{\Omega_R^{\rm out}}	\frac{Q_{R, i}(v/R)^2}{f_R}\d v + \int_{\Omega_R^{\rm out}} \frac{| \mathcal E_R^Q(v) |^2}{f_R}\d v
	\end{align*}
    Explicit forms of $Q_{R,i}$ in~\eqref{eq:out_leading_terms} and the error bound in~\eqref{eq:out_error_est} show that
	\begin{align*}
	   \int_{\Omega_R^{\rm out}}	\frac{Q(f_R,f_R)^2}{f_R}\d v \leq CR^{-K-6}.
    \end{align*}
    \textbf{Step 2: Contribution of $Q(f_R,f_R) \delta_{f_R}$.}  We combine the expansions of $Q(f_R,f_R)$ and $\delta_{f_R}$ to rewrite 
    \begin{align*}
	   \int_{\Omega_R^{\rm out}} Q(f_R,f_R)\,\delta_{f_R} \d v  = 	\int_{\Omega_R^{\rm out}} \lp Q_{R,1} \delta_{R,1}  \rp (v/R) \d v + \sum_{\substack{i \in \{1,2,3\},  \\j \in \{1,2,3,4\}, \\ (i,j) \neq (1,1)}} \int_{\Omega_R^{\rm out}} \lp Q_{R,i} \delta_{R,j}  \rp (v/R)\d v \\
      +\sum_{j=1}^4 \int_{\Omega_R^{\rm out}}
    	\mathcal E_R^Q(v)
    	\delta_{R,j}(v/R)
    	\d v
    	+
    	\sum_{i=1}^3\int_{\Omega_R^{\rm out}}
    	Q_{R,i}(v/R) \mathcal E_R^\delta(v)
    	\d v
    	+ 
    	\int_{\Omega_R^{\rm out}}
    	\mathcal E_R^Q(v)\mathcal E_R^\delta(v)
    	\d v.
    \end{align*}
	The first term is given by 
	\begin{align*}
		\int_{\Omega_R^{\rm out}}
		\lp Q_{R,1} \delta_{R,1}  \rp (v/R) \d v &=   2R^{-K-7}	\int_{\Omega_R^{\rm out}} \frac{\Pi^\perp_v}{|v/R|^2}:\nabla^2G(v/R) \d v \\&= 2 R^{-K-4} \int_{ \left \{|v|> R^{-\frac25} \right \}}\frac{\Pi^\perp_v}{|v|^2}:\nabla^2G(v) \d v.
	\end{align*}
	By~\cref{lem:prop_G}, we obtain the leading order term 
	\begin{align*}
		\int_{\Omega_R^{\rm out}}
		\lp Q_{R,1} \delta_{R,1}  \rp (v/R) \d v = \frac{8}{3}R^{-K-4}	+O \lp \exp \lp -cR^\frac45 \rp \rp.
	\end{align*}
	Since $K>2$, all remaining terms are of higher order. More precisely, 
    \begin{align*}
    	\sum_{\substack{i \in \{1,2,3\},  \\ j \in \{1,2,3,4\}, \\(i,j) \neq (1,1)}} \int_{\Omega_R^{\rm out}} \lp Q_{R,i} \delta_{R,j}  \rp (v/R)\d v = O (R^{-K-6}).
    \end{align*}
Moreover, using the error estimates~\eqref{eq:out_error_est} and~\eqref{delta_gR_error}, together with the Cauchy--Schwarz inequality, all products containing at least one error term satisfy
    \begin{align*}
    	\int_{\Omega_R^{\rm out}}  \sum_{j=1}^4 
    	\left|
    	\mathcal E_R^Q(v)
    	\delta_{R,j}(v/R)
    	\right|
    	+
    	\sum_{i=1}^3
    	\left| Q_{R,i}(v/R)
    	\mathcal E_R^\delta(v)
    	\right|
    	+
    	\left|
    	\mathcal E_R^Q(v)\mathcal E_R^\delta(v)
    	\right| \d v =
    	O\lp\exp\lp-cR^{\frac45}\rp\rp.
    \end{align*}
\end{proof}

\subsection{Proof of~\cref{thm:counter_example}}
\label{sec:proof_end}
\begin{proof}
	We split the exact identity \eqref{eq:exact-derivative-Q-form} over the three
	regions $\Omega_R^{\rm in}$, $\Omega_R^{\rm mid}$, and
	$\Omega_R^{\rm out}$
	\begin{align*}
		\frac{\d}{\d t}D_H(f_R(t)) 	  =& \int_{\Omega_R^{\rm in}}Q(f_R,f_R)\,\delta_{f_R}\d v	-2\int_{\Omega_R^{\rm in}}\frac{Q(f_R,f_R)^2}{f_R}\d v \\ &+ \int_{\Omega_R^{\rm mid}}Q(f_R,f_R)\,\delta_{f_R}\d v	-2\int_{\Omega_R^{\rm mid}}\frac{Q(f_R,f_R)^2}{f_R}\d v  \\&+ \int_{\Omega_R^{\rm out}}Q(f_R,f_R)\,\delta_{f_R}\d v	-2\int_{\Omega_R^{\rm out}}\frac{Q(f_R,f_R)^2}{f_R}\d v.
	\end{align*}
We apply~\cref{prop:inner-contribution} to the first line,~\cref{prop:middle-contribution} to the second line and~\cref{prop:out-contribution} to the last line to obtain
\begin{align*}
		\left.\frac{\d}{\d t}D_H(f(t))\right|_{t=0}
	=\frac{8}{3}R^{-K-4}	+O(R^{-K-6})	+O(R^{-2K-2}) + O \lp \exp \lp -cR^\frac45 \rp \rp = \frac{8}{3}R^{-K-4} + o (R^{-K-4}),
\end{align*}
where in the last step we used $K>2$. 

 The derivative is thus	strictly positive for all sufficiently large $R$.  By continuity along the
	classical solution, $D_H(f_R(t))$ is strictly increasing for sufficiently
	small positive times.
\end{proof}

	\section*{Acknowledgements}
	J.Junn\'e is supported by the French National Research Agency (ANR) through the project KEN (ANR-22-CE40-0016) \href{https://anr.fr/Projet-ANR-22-CE40-0016}{\includegraphics[height=\fontcharht\font`\B]{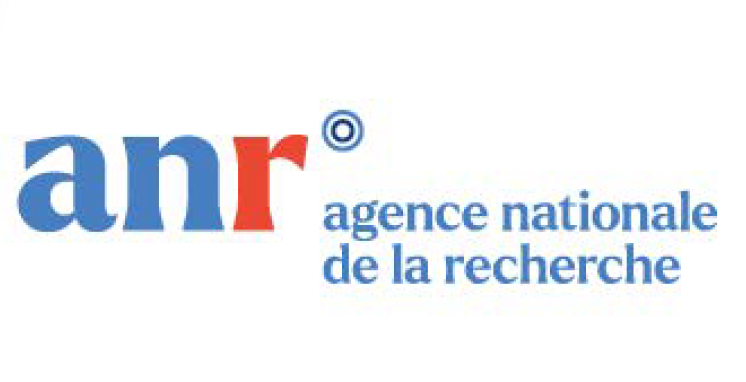}}. H. Yolda\c{s} is supported by the Dutch Research Council (NWO) \href{https://www.nwo.nl/en/projects/viveni222288}{\includegraphics[height=\fontcharht\font`\B]{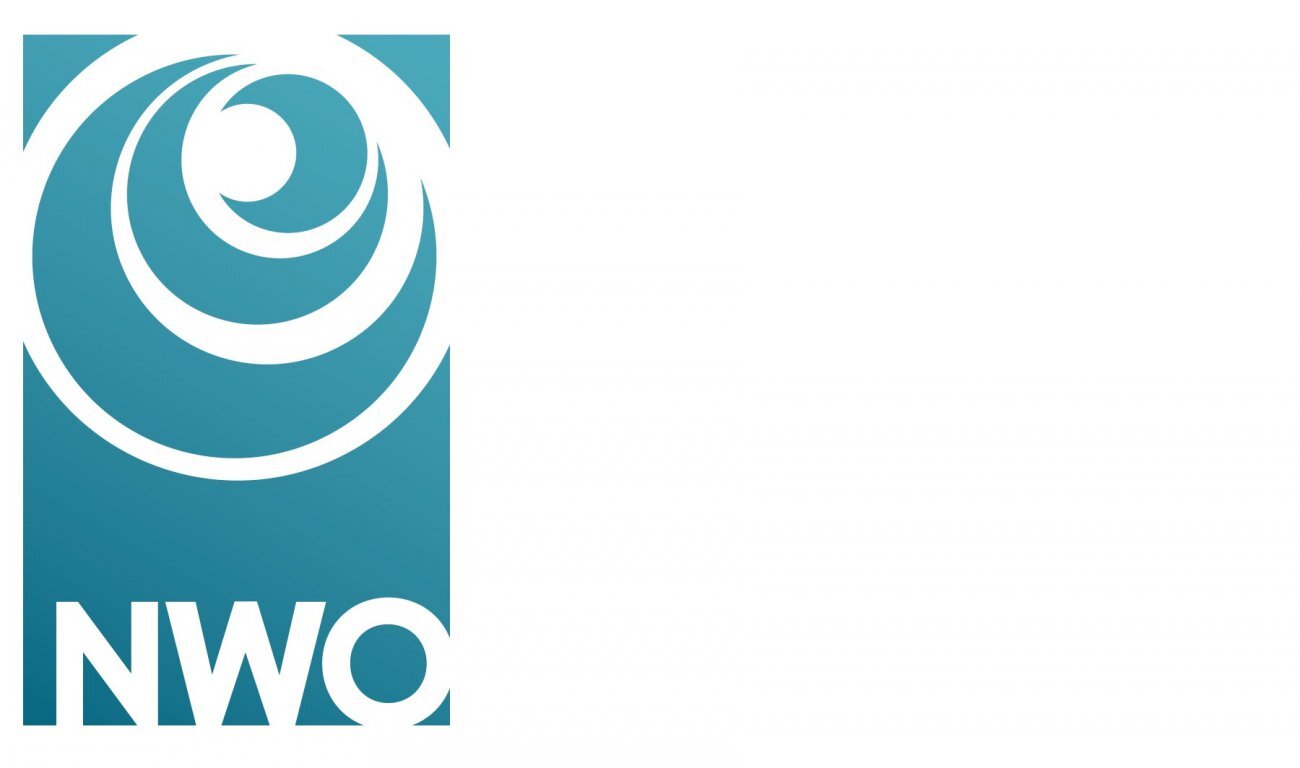}} \hspace{-10pt} under the NWO-Talent Programme Veni ENW project MetaMathBio -- VI.Veni.222.288.

	\section*{Declarations}
	
	\paragraph*{Data availability.} Data sharing is not applicable to this article as no datasets were generated or analyzed during the study.  
	\paragraph*{Conflict of interest.}
	The authors declare no conflicts of interest regarding this manuscript.
	\medskip 
	
	\noindent For the purpose of open access, the authors have applied a Creative Commons Attribution (CC-BY) licence to any Author Accepted Manuscript version arising from this submission.
	
	\paragraph*{AI Disclosure.} The specific sequence $f_R$ used as a counterexample was found with the help of AI. The proof given by AI was rather convoluted and has been rewritten by the authors. The authors are responsible for the mathematical arguments and the content of the paper. 
	
	\appendix
	\section{Appendix}\label{sec:appendixA}
	
	\begin{proposition}[Diffusion matrix at equilibrium]
		\label{prop:core-diffusion-expansion}
		There exist constants $c,C>0$ such that,
		\begin{equation}
			\label{eq:core-diffusion-expansion}
			A[M](v) 	=
			\frac{1}{|v|}\Pi^\perp_v +\frac{1}{|v|^3}	\left(	3\frac{v\otimes v}{|v|^2}-\Id	\right) +	\mathcal E_A(v)
			\qquad \text{for $|v |>1$},
		\end{equation}	where
	\begin{align*}
			|\mathcal E_A(v)|	\leq	C \exp \lp -c|v|^2 \rp.
	\end{align*}
	In particular,
		\begin{equation}
			\label{eq:core-radial-contraction}
			A[M](v) :(v\otimes v)	=	\frac{2}{|v|}+ 	O  \lp \exp \lp -c|v|^2\rp \rp.
		\end{equation}
	\end{proposition}
	
	\begin{proof}
	By rotational symmetry of the Maxwellian distribution
		\begin{align*}
		A[M] (v)= \lambda_M^\parallel(r)\frac{v\otimes v}{r^2} +
		\lambda_M^\perp(r)\Pi^\perp_v
		\end{align*} 
	where we used the notation $r=|v|$ and $\lambda_M^\parallel, \lambda_M^\perp$ are radial functions. By~\cite[Proppositon 3.2]{GG16} we have 
	\begin{align*}
		\lambda_M^\parallel(r)	=	\frac{2}{3r^3}	\int_{|w|<r}|w|^2M(w)\d w	+
		\frac23 \int_{|w|>r}\frac{M(w)}{|w|}\d w.
	\end{align*}
	Since $M$ is Maxwellian with unit temperature, we have
\begin{align*}
	\lambda_{M}^\parallel(r)=	\frac{2}{r^3}	-
	\frac{2}{3r^3}	\int_{|w|>r}|w|^2M(w)\d w	+
	\frac23	\int_{|w|>r}\frac{M(w)}{|w|}\d w.
\end{align*} 
Then, due to the Gaussian tail of the Maxwellian we obtain
\begin{align*}
		\lambda_{M}^\parallel(r)	=	\frac{2}{r^3}	+
	O \lp \exp \lp -cr^2 \rp \rp.
\end{align*}
To find the corresponding expansion for $\lambda_M^\perp$, we use the trace identity
\begin{align}\label{trace_identity}
		\lambda_M^\parallel(r)+2\lambda_M^\perp(r)=\mbox{Tr}(A[M])(v)	=	2 \int_{\R^3}	\frac{M(w)}{|v-w|} \d w. 
\end{align}
	The latter can be expressed as 
	\begin{align*}
		\int_{\R^3}	\frac{M(w)}{|v-w|} \d w = \frac1r\int_{|w|<r}M(w)\d w
		+\int_{|w|>r}\frac{M(w)}{|w|}\d w.
	\end{align*}
	Inserting this into~\eqref{trace_identity} gives 
	\begin{align*}
			\lambda_{c_R}^\perp(r)	=	\frac{1}{r}	-
		\frac{1}{r^3} 	+ 	O\!\left(e^{-cr^2}\right).
	\end{align*}
This finishes the proof. 
	\end{proof}

\begin{lemma}
	\label{lem:prop_G}
	The function $G$ defined in~\eqref{def:G} satisfies
	\begin{equation}\label{eq:profile-sign-identities}
		\int_{\R^3}\frac{G(y)}{|y|^4}\d y
		=
		\frac23,
		\qquad
		\int_{\R^3}\frac{\Pi_y^\perp:\nabla^2G(y)}{|y|^2}\d y
		=
		2\int_{\R^3}\frac{G(y)}{|y|^4}\d y
		=
		\frac43.
	\end{equation}
\end{lemma}

\begin{proof}
	Set
	\begin{align*}
		J_p
		\coloneqq
		\int_0^\infty r^p \exp \lp -r^2-r^{-2} \rp  \d r.
	\end{align*}
	The change of variables $r\mapsto r^{-1}$ gives $J_0=J_{-2}$.  Moreover,
	integrating the derivative of $r \exp \lp -r^2-r^{-2} \rp$ gives
	$3J_{-2}=2J_2$.  Since $Z=4\pi J_2$,
	\begin{align*}
			\int_{\R^3}\frac{G(y)}{|y|^4}\d y	=	\frac{4\pi J_{-2}}{Z}
		=	\frac23.
	\end{align*}
	For the second identity, write $G(y)=g(r)$ with $r=|y|$.  Radiality gives
\begin{align*}
	\frac{\Pi_y^\perp}{|y|}:\nabla^2G(y)
	=
	\frac{2g'(r)}{r^2}.
\end{align*}
	Thus, using the flatness of $g$ at the origin and its Gaussian decay at
	infinity,
	\begin{align*}
		\int_{\R^3}\frac{\Pi_y^\perp:\nabla^2G(y)}{|y|^2}\d y
		=
		8\pi\int_0^\infty\frac{g'(r)}r\d r=
		8\pi\int_0^\infty\frac{g(r)}{r^2}\d r=
		2\int_{\R^3}\frac{G(y)}{|y|^4}\d y.
	\end{align*}
	This concludes the proof.
\end{proof}
	
	\printbibliography
	
\end{document}